\documentclass[12pt]{article}
\usepackage{amsthm,amsmath,amssymb,amscd,graphics,enumerate,latexsym,verbatim}
\usepackage[all]{xy}
\usepackage{ulem}	

\makeatletter
\renewcommand*\@fnsymbol[1]{\the#1}
\makeatother

\title{On the Algebraic K-theory of Monoids}
\author{Chenghao Chu\footnote{Department of Mathematics, Johns Hopkins University, 3400 N. Charles Street, Baltimore MD 21218, USA. cchu@math.jhu.edu},  Jack Morava\footnote{Department of Mathematics, Johns Hopkins University, 3400 N. Charles Street, Baltimore MD 21218, USA. jack@math.jhu.edu} }

\newtheorem{theorem}{Theorem}[section]

\newtheorem{corollary}[theorem]{Corollary}

\newtheorem{definition}[theorem]{Definition}
\newtheorem{example}[theorem]{Example}

\newtheorem{lemma}[theorem]{Lemma}

\newtheorem{proposition}[theorem]{Proposition}
\newtheorem{remark}[theorem]{Remark}

\ifx\pdfoutput\@undefined\usepackage[usenames,dvips]{color}
\else\usepackage[usenames,dvipsnames]{color}
\IfFileExists{pdfcolmk.sty}{\usepackage{pdfcolmk}}{}
\fi
\definecolor{chu}{rgb}{0,0.5,0}
\definecolor{error}{rgb}{0.8,0,0}
\definecolor{oliver}{rgb}{0,0,0.8}
\definecolor{rekha}{rgb}{0.7,0.4,0.6}
\definecolor{k}{rgb}{.5,0,.5}

\usepackage{pdfsync}

\numberwithin{equation}{section}

\DeclareMathOperator{\Spec}{Spec}

\def\F{{\mathbb F}}

\def\Z{{\mathbb Z}}

\def\cM{{\mathcal M}}

\def\cS{{\mathcal S}}

\def\Fun{{\F_1}}

\def\Mo{{\mathcal M_0}}

\def\int{\textup{int}}

\begin{document}

\maketitle \setcounter{footnote}{1}

\begin{abstract} Let $A$ be a not necessarily commutative monoid with zero such that projective $A$-acts are free. This paper shows that the algebraic K-groups of $A$ can be defined using the +-construction and the Q-construction. It is shown that these two constructions give the same K-groups. As an immediate application, the homotopy invariance of algebraic K-theory of certain affine $\Fun$-schemes is obtained. From the computation of $K_2(A),$ where $A$ is the monoid associated to a finitely generated abelian group, the universal central extension of certain groups are constructed.
\end{abstract}


\section{Introduction}\label{Recollections and Notations}
The theory of algebraic geometry over $\Fun$ has recently been developed by many authors.  Deitmar \cite{D05} constructed $\Fun$-schemes using monoids in the same way as the usual schemes are constructed using commutative rings. \cite{D05} also studied the category of quasi-coherent sheaves over these schemes.  For example, an affine $\Fun$-scheme as defined in \cite{D05} is the spectrum of a commutative monoid  and the category of quasi-coherent sheaves over this affine $\Fun$-scheme is equivalent to the category of acts over the corresponding monoid. Please see \cite{D05} for more details. Connes and Consani \cite{CC09} modified and generalized Deitmar's definition of $\Fun$-schemes. One of the necessary modifications is to require monoids to have the zero element $0$.  It is known that Deitmar's $\Fun$-schemes are special type of Connes and Consani's $\Fun$-Schemes, which are called $\Mo$-schemes. In this paper, by ``$\Fun$-scheme" we always mean the $\Fun$-scheme as defined in \cite{CC09} and by ``monoid'' we mean a monoid with $0$ but it may not necessarily be commutative.

Using the terminology of $\Fun$-schemes, K-theory of projective space over monoids were considered in \cite{H-W01, H02}. The algebraic K-theory of affine $\Mo$-schemes, or equivalently commutative monoids, was considered in \cite{D06} via the +-construction (denoted as $K^+$) and Q-construction of Quillen (denoted as $K^Q$). However, $K^+\neq K^Q$ in general because projective acts over a monoid are not direct summands of free acts in general.
In \cite{COR}, the category of quasi-coherent sheaves of modules over $\Fun$-schemes has been generalized from $\Mo$-schemes to $\Fun$-schemes. This category allows one to defined algebraic K-theory of $\Fun$-schemes via  Q-construction. For example, if $G$ is an abelian group and let $\Fun[G]$ denote the monoid $G\cup \{0\}$ and let $Spec(\Fun[G])$ be the spectrum of $\Fun[G]$.  Then the K-groups of $Spec(\Fun[G])$ are naturally isomorphic to the stable homotopy groups of $BG+$ $$K_i(Spec(\Fun[G]))\cong \pi_i^s(BG_+),$$ which is widely believed to be true.

To obtain more properties and computational results on the algebraic K-theory of $\Fun$-schemes, we continue to study the K-theory of not necessarily commutative monoids in this paper. In Section \ref{InvertibleMatricesOverMonoids}, we check that the +-construction can be applied to noncommutative monoid $A$. This is done by showing that there is a reasonable notion of stable special linear group $E(A)$ of the widely accepted general linear group $GL(A)$. From the viewpoint of $\Fun$-geometry, an $A$-set is  an $A$-act with a base point and the zero element of $A$ acts as the constant map sending every element to the base point. This makes the category of $A$-sets slightly different from the category of $A$-acts. So we feel necessary to study the category of $A$-sets in details in Section \ref{ModulesOverMonoids}. In particular, we check that the notion of admissible exact sequences of $A$-sets where $A$ is commutative (see \cite{COR}) can be generalized to the case that $A$ is not necessarily commutative. This means one can use Q-construction to define algebraic K-theory of monoids. If all projective $A$-sets are free, one can prove the $``Q=+"$ theorem as usual.  As the first application, we obtain in Section \ref{TheComparisonTheorem} the homotopy invariance of algebraic K-theory of the affine $\Fun$-scheme $\Spec(A)$, where $A$ is commutative and projective $A$-sets are free. As another application, in Section \ref{UniversalCentralExtension} we compute $\pi_2^s(BG_+)$ via universal central extensions for $G$ finitely generated and abelian.

\section{Invertible matrices over monoids}\label{InvertibleMatricesOverMonoids}
Algebraic K-theory of commutative monoids via +-construction has been discussed briefly in \cite{D06}. In this section, we provided more details and check that the same construction applies to noncommutative monoids.

 By {\it monoid} we always mean a not necessarily commutative monoid with an absorbing element $0$ and a unit $1$. Recall that an element $0$ in a monoid $A$ is called absorbing if $0a=a0=0$ for all $a\in A$. If $0=1$, then $A$ is the zero monoid $\{0\}$. If $A$ has an absorbing element, then the absorbing element is unique. Morphisms of monoids are maps of sets which are multiplicative and preserve $0$ and $1$. Following \cite{CC09}, we think monoids as  $\Fun$-algebras.

\begin{example} If $R$ is a ring, then $R$ is a monoid under the multiplication.
\end{example}

If $A$ is a monoid, let $\Z A$ denote the usual monoid ring whose underlying additive group is the free abelian group on the set $A$. The base extension of $A$, denoted as $\Z[A]$, is defined as the quotient ring of $\Z A$ by the ideal $\Z\cdot 0$ where $0$ is the absorbing element in $A$.

\begin{example}\label{example:Monoids} If $A$ is a monoid which may or may not have an absorbing element, let $A_*=A\cup \{*\}$ be the monoid by adding a distinguished point $*$ to $A$ and requiring that $*\cdot x=x\cdot *=*$ for any $x\in A_*$. Then $A_*$ has an absorbing element $*$. The monoid ring of $A_*$ is isomorphic to the usual monoid ring of $A$. That is, $ \Z[A_*]\cong \Z A$.

Sometimes, we also denote the monoid $A_*$ by $\Fun[A]$.
\end{example}

 Now we define  matrices over $A$, following \cite{K76}. The idea is that a matrix over $A$ defines an A-map between free $A$-sets, which are explained in more details in Section \ref{TheCategoryOfModulesOverMonoids}.
 Let $m,n$ be positive integers. The set of $m\times n$ matrices, denoted as $M_{m\times n}(A)$, is a subset of all $m\times n$ matrices over $\Z[A]$, denoted as $M_{m\times n}(\Z[A])$. A matrix $M=(m_{ij})\in M_{m\times n}(\Z[A])$ is in $M_{m\times n}(A)$ if and only if all entries are in $A$ and $m_{ij}=0$ if $f(i)\neq j$ for some maps of sets $$f: \{1,2,\cdots, m\}\longrightarrow \{1,2,\cdots, n\}.$$

 \begin{lemma}\label{lemma:MatrixMultiplicationOverMonoids} The restriction of the multiplication of matrices over the ring $\Z[A]$ defines the multiplication
 $$M_{m\times n}(A) \times M_{n\times k}(A) \longrightarrow M_{m\times k}(A).$$
 \end{lemma}

 \begin{proof} Direct computations shows that the multiplication of two matrices over $A$ is again a matrix over $A$.
 \end{proof}

 If $m=n$, we simply denote $M_{m\times n}(A)$ as $M_n(A)$. Similar to \cite{K76}, elements in $M_n(A)$ are called {\it row monomic matrix}. $M_n(A)$ is a monoid under multiplication and we denote the group of invertible elements by $GL_n(A)$. A matrix in $GL_n(A)$ can be written as $D(a_1, \cdots, a_n)\sigma$ where $\sigma$ is a permutation matrix and $D(a_1,\cdots, a_n)$ is a diagonal matrix whose diagonal entries are invertible elements $a_1, a_2,\cdots, a_n\in A.$ There is a natural inclusion of groups $GL_n(A)\to GL_{n+1}(A)$ which enable us to define the stable general linear group over $A$ as
 $$GL(A)=\bigcup\limits_{i=1}^\infty GL_n(A).$$

Let $E(A)$ be the normal subgroup of commutators $[GL(A), GL(A)]$. Recall that a group $G$ is called {\it perfect} if $G=[G, G]$. It is clear to experts that $E(A)$ is a perfect group. To be self-contained, we now provide a proof of this fact.
\begin{proposition}\label{prop: E(A)IsPerfect} Let $A^\times$ denote the group of invertible elements of a monoid $A$.
\begin{enumerate} \item[(1)] An $n\times n$ matrix $M\in E(A)$ if and only if $M=D(a_1, \cdots, a_n)\sigma$ where $\sigma$ is an even permutation and $a_1a_2\cdots a_n\in [A^\times, A^\times]$.
\item[(2)] $E(A)$ is a perfect group.
\end{enumerate}
\end{proposition}
\begin{proof} Let $G$ be the subset of $GL(A)$ which contains matrices of the form $D(a_1, \cdots, a_n)\sigma$ where $a_1a_2\cdots a_n\in [A^\times, A^\times]$ and $\sigma$ is an even permutation for some $n$. Using the formula
$$D(a_1, \cdots, a_n)\sigma\cdot D(b_1, \cdots, b_n)\alpha=D(a_1, \cdots, a_n)\cdot (\sigma D(b_1, \cdots, b_n)\sigma^{-1})\cdot(\sigma\alpha) $$
and the fact that $$\sigma D(b_1, \cdots, b_n)\sigma^{-1}=D(b_{\sigma(1)}, \cdots, b_{\sigma(n)})$$ one shows that $G$ is a subgroup. Using a similar calculation one also checks that $E(A)$ is a subgroup of $G$ and that $G$ is normal in $GL(A)$. To show that $E(A)=[GL(A), GL(A)]=G$ and to show that $E(A)$ is perfect, we only need to check that $G$ is a subgroup of $[E(A), E(A)]$.

It is well know that all even permutations are in $[E(A), E(A)]$ since $E(A)=[GL(A), GL(A)]$ contains all even permutations. So it is enough to show that any diagonal matrices $D(a_1, \cdots, a_n)$ is in $[E(A), E(A)] $,  where $a_1a_2\cdots a_n\in [A^\times, A^\times]$. If $a\in A^\times$, let $D_{ij}(a)$ be the diagonal matrix whose $(i,i)$-entry is $a$ and $(j,j)$-entry is $a^{-1}$ and $1$ for all other diagonal entries. We have
$$D(a_1, \cdots, a_n)=D(a_1\cdots a_n,1,\cdots, 1)D_{1n}(a_n^{-1})\cdots D_{12}(a_2^{-1}).$$
So we have to show that all matrices on the right hand side of the above equation belong to $[E(A), E(A)]$. Since $[E(A), E(A)]$ is also normal in $GL(A)$ and $D_{ij}(a)=\sigma D_{12}(a)\sigma^{-1}$ for some suitable permutation matrix $\sigma$, we only need to show that the following two matrices are in $[E(A), E(A)]$.
$$D(aba^{-1}b^{-1}, 1)=\left(\begin{array}{cc}aba^{-1}b^{-1}& 0\\0&1\end{array}\right), \quad  D(a, a^{-1})=\left(\begin{array}{cc}a & 0\\0& a^{-1}\end{array}\right).$$

The matrix $D(a, a^{-1})$ is in $E(A)$ because $$D(a, a^{-1})=\left(\begin{array}{cc}a & 0\\0& a^{-1}\end{array}\right)=
\left(\begin{array}{cc}a& 0\\0&1\end{array}\right)
\left(\begin{array}{cc}0& 1\\1& 0\end{array}\right)
\left(\begin{array}{cc}a^{-1}& 0\\0&1\end{array}\right)
\left(\begin{array}{cc}0& 1\\1& 0\end{array}\right).$$
Let $\sigma\in E(A)$ be the permutation matrix which corresponds to the even permutation $(123)(456)$. We have the following direct calculation
$$D(a, a^{-1})\sigma D(a, a^{-1})\sigma^{-1}=D(a, 1, a^{-1})$$
which shows $D(a, 1, a^{-1})\in [E(A), E(A)]$.  Since $[E(A), E(A)]$ is normal and $D(a, 1, a^{-1})$ is conjugate to $D(a, a^{-1})$, we see that $D(a, a^{-1})$ is in $[E(A), E(A)]$.

The following standard calculation shows $D(aba^{-1}b^{-1}, 1)\in [E(A), E(A)]$ and hence completes the proof.
$$D(aba^{-1}b^{-1}, 1)=D(a,a^{-1}, 1)D(b, 1, b^{-1})D(a^{-1}, a, 1)D(b^{-1}, 1, b)$$
\end{proof}

Recall the +-construction of Quillen from \cite{Srinivas}. Let $(X, x)$ be a pointed path connected space that has the homotopy type of a CW complex. Let $G$ be a normal perfect subgroup of $\pi_1(X,x)$. Quillen showed that there exists a unique, up to homotopy,  pointed space $(X^+,x^+)$ together with a map
$f\colon (X,x)\longrightarrow (X^+,x^+)$ satisfying a list of properties, among which it says that the induced map on the fundamental groups is onto with kernel $G$. If $A$ is a monoid, Proposition \ref{prop: E(A)IsPerfect} shows that $E(A)$ is a  normal  perfect subgroup of $GL(A)$. So we can define the algebraic K-theory of a monoid using the $+$-construction. The following definition was given in \cite{D06} in the case that $A$ is commutative.
 \begin{definition}
 Let $A$ be a monoid with an absorbing element. For any integer $i\geq 1$, define the $i$-th algebraic $K$-group of $A$ as
 $$K_i(A)=\pi_{i}(BGL(A)^+),\quad \quad i\geq 1,$$ where $BGL(A)^+$ is the +-construction of $BGL(A)$ with respect to the normal perfect subgroup $E(A)$. If $A$ does not have an absorbing element, let $A_*$ be the monoid which is obtained by adding a distinguished element as the zero element. Define $K_i(A)=K_i(A_*)$ for any $i\geq 1$.
 \end{definition}

The following fact is clear from the definition. One can refer \cite{So} for a detailed proof.
\begin{corollary} $K_1(A)\cong \Z/2\oplus (A^\times)^{ab}$ where $(A^\times)^{ab}=A^\times/[A^\times, A^\times]$ is the abelianization of $A^\times$.
\end{corollary}
The following fact is also standard. One may refer \cite{Srinivas} for a proof.
\begin{corollary}\label{corr:K2IsTheKernel} $K_2(A)$ is the kernel of the universal central extension of $E(A)$.
\end{corollary}

\begin{proposition}\label{prop:UnitsOfMonoidsCompletelyDeterminsKGroups} If $A$ and $B$ are monoids and there is an isomorphism $f: A^\times \cong B^\times$, then $f$ induces isomorphisms $K_i(A)\cong K_i(B)$ for any $i\geq 1.$
\end{proposition}
\begin{proof} This is because $f$ induces an isomorphism $GL(A)\cong GL(B)$.
\end{proof}

\begin{corollary} If $A$ has an absorbing element, $K_i(A)\cong K_i(A_*)$ for any $i\geq 1$.
\end{corollary}
\begin{proof} This statement follows from Proposition \ref{prop:UnitsOfMonoidsCompletelyDeterminsKGroups} and the fact that $A$
 and $A_*$ has the same group of invertible elements.
\end{proof}

\section{Modules over monoids}\label{ModulesOverMonoids}

\subsection{The category of modules over monoids}\label{TheCategoryOfModulesOverMonoids}
The constructions in this subsection are standard and it would be no harm for not interested readers to skip this subsection.
Let $A$ be a monoid. If $A$ is commutative, the category of $A$-sets was defined in \cite{COR}. In this subsection, we generalize the constructions to noncommutative monoids. One can also regard this subsection as a continuation of \cite{K72} to the case that both monoids and acts are based.

Let $\cS_*$ denote the category of pointed sets. If $M$ is a pointed set, then $Hom_{\cS_*}(M, M)$ is a monoid with absorbing element the constant map which has $\{*\}$ as the image. We  write $0$ for the element $*$ if there is no confusion.
 \begin{definition}
 A pointed set $M$ is called a left $A$-set if there is a monoid morphism $A\to Hom_{\cS_*}(M, M)$ and $M$ is a right $A$-set if there is a monoid morphism $A^{op}\to Hom_{\cS_*}(M, M)$ where $A^{op}$ is the opposite monoid of $A$.
\end{definition}

In this paper, we only consider left $A$-sets. Let $M$ be an $A$-set. For any $a\in A$ and $m\in M$, we write $am$ for the image of $m$ under the map $a$. Let $Am$ be the subset $\{am|a\in A\}$, it is a submodule of $M$. If there exist  finitely many elements $m_1, \cdots, m_n\in M$ such that $M=Am_1\cup \cdots \cup Am_n$, then $M$ is said to be {\it finitely generated}. An equivalence relation ``$\sim$'' on $M$ is called a {\it congruence} if $x\sim y$ implies $ax\sim ay$ for any $a\in A$. If $\sim$ is a congruence on $M$, the quotient set $M/\sim$ is defined as  usual. That is, the underlying set of $M/\sim$ is the set of equivalence classes of $\sim$. If we use $\overline{m}$ to denote the equivalence class of $m$, then we define $a\overline{m}=\overline{am}$ for any $a\in A$. The quotient map from $M$ to $M/\sim$ is an $A$-set morphism.

Let $M$ and $N$ be $A$-sets. A map of based sets $f:M\to N$ is called an $A$-set morphism or simply morphism if $f(am)=af(m)$ for all $a\in A$ and $m\in M$. Let $A-\mathcal{S}et$ denote the category of left $A$-sets and $A$-set morphisms. The object $*$ is both an initial and terminal object in $A-\mathcal{S}et$. The {\it kernel} and {\it cokernel} of a morphism $f:M\to N$ is defined as equalizer and coequalizer of the following diagram, respectively.
$$\xymatrix{ M \ar@<.5ex>[r]^f  \ar@<-.5ex>[r]_{*}
 &  N
}$$
Since the objects of $A-\mathcal{S}et$ are sets, we  identify kernel (cokernel) as sub objects (quotient objects) rather than abstract morphisms.

\begin{lemma}\label{lemma:EqualizaorAndCoequalizerOfAModulesExist} In the category $A-\mathcal{S}et$, equalizer and coequalizer of any diagram $\xymatrix{ M \ar@<.5ex>[r]^f  \ar@<-.5ex>[r]_{g}
 &  N
}$ exist.  In particular, all morphisms  have kernel and cokernel.
\end{lemma}

\begin{proof}
The equalizer is the submodule $\{ m\in M \big\vert f(m)=g(m)\}$. In particular, $f^{-1}(*)$ is the kernel of $f$.

Elements $x,y\in N$ are said to be direct related by $f$ and $g$ if any of the following statements is true.
\begin{enumerate}\item[1)] $x=y$.
\item[2)] $x=f(m)$ and $y=g(m)$ for some $m\in M$.
\item[3)] $x=g(m)$ and $y=f(m)$ for some $m\in M$.
\end{enumerate}
Note that $x$ is directly related to $y$ implies $ax$ being directly related to $ay$ for any $a\in A$. Define an equivalent relation $\sim$ in $N$ as follows. For any $x, y\in N$,  $x$ is related to $y$ if and only if there exists finitely many points $x=x_0, x_1, x_2, \cdots, x_n=y$ such that $x_i$ is directly related to $x_{i+1}$ for any $0\leq i \leq (n-1)$. One checks that $\sim$ is a congruence on $N$ and the coequalizer of $f$ and $g$ is the quotient $A$-set $A/\sim$.

Since $f(M)$ is a submodule of $N$ and  any submodule defines a congruence in a natural way, the quotient set $N/f(M)$ is defined. The underlying set of $N/f(M)$ is the set $\big(N-f(M)\big) \cup \{*\}$. One checks that $N/f(M)$ is the cokernel of $f$.
\end{proof}

\begin{remark} It follows from the proof of the previous lemma that cokernel of a morphism $f$ being zero implies that $f$ is onto. But kernel of $f$ being zero does not imply that $f$ is injective.
\end{remark}

\begin{lemma}\label{lemma:ProductAndCoproductExist} Arbitrary product and arbitrary coproduct exist in $A-\mathcal{S}et$.
\end{lemma}
\begin{proof}
Product is given by the usual cartesian product $\prod$ and the coproduct is given by the usual wedge product $\bigvee$.
\end{proof}

\begin{corollary} The category $A-\mathcal{S}et$ is complete and cocomplete.
\end{corollary}

\begin{remark} As we can see from the constructions in Lemma \ref{lemma:EqualizaorAndCoequalizerOfAModulesExist} and Lemma \ref{lemma:ProductAndCoproductExist}, the underlying pointed set of a limit or colimit of $A$-sets can be calculated as
 the same limit or colimit of the underlying based set of $A$-sets. In other words, limit and colimit commutes with the forgetful functor from
 $A$-sets to pointed sets.
\end{remark}

\begin{example} Let $\Fun$ be the monoid $\{0, 1\}$. The category of $\Fun$-modules is the category of pointed sets, which is well known to be complete and cocomplete.
\end{example}

Let $A$ be a monoid, then $A-\mathcal{S}et$ embeds as a subcategory of $\Fun-\cM od$ by forgetting the $A$-set structure. This forgetful functor has a left adjoint, which we denote by $F$ for the moment. For a pointed set $X$, this discussion justifies that we can call $F(X)$ the {\it free $A$-set} generated by $X$ (or more appropriately by $X-\{*\}$). It is obvious that
$$F(X)\cong \bigvee\limits_{x\in X-\{*\}}Ax.$$

We now study the tensor functor and the $Hom$-functor. If $A$ and $B$ are monoids, an {\it $(A,B)$-biset} is a set $M$ which is a left $A$-set and also a right $B$-module such that $(am)b=a(mb)$ for any $a\in A, m\in M$ and $b\in B$. Here we adopt the usual convention to write $b(m)$ as $mb$ since $M$ is a right $B$-module.  If $M$ is an $(A,B)$-biset and $N$ is a $(B, C)$-biset, there is an equivalence relation on the set theoretic cartesian product $M\times N$  which is generated by $(mb,n)\sim (m, bn)$. The set $M\times N$ is naturally an $(A, C)$-biset and the equivalence is a congruence with respect to both $A$ and $C$. So the set of equivalent classes is an $(A, C)$-biset, which is defined as the tensor product $M\otimes_BN$. The following lemma states some standard properties of the tensor product of sets over monoids. These properties are well known for tensor products of modules over rings. The proof of the lemma is also the same as in the case of modules over rings and hence skipped.

\begin{lemma}Let $A, B, C$ and $D$ be monoids. Let $M$ be an $(A,B)$-biset,  $N$ a $(B, C)$-biset and $P$ a $(C, D)$-biset.
\begin{enumerate}
\item[(1)] Associative. $(M\otimes N) \otimes P \cong M\otimes (N\otimes P)$ as $(A, D)$-bisets.
\item[(2)] Symmetric. If $A$ is commutative and $M$ and $N$ are $A$-sets, then $M\otimes N\cong N\otimes M$ as $A$-sets.
\item[(3)] Functorial. Tensor product $-\otimes-$ is a bifunctor.
\end{enumerate}
\end{lemma}

If $M$ is an $(A, B)$-biset and $P$ is an $(A, C)$-biset, then the set of $A$-set morphisms $Hom_A(M, P)$ is a $(B, C)$-biset. If $f\in Hom_A(M,P)$, then $bfc$ is the map sending $m\in M $ to $f(mb)c\in P$. Similarly, if $N$ is a $(B, C)$-biset, then $Hom_C(N, P)$ is an $(A,B)$-biset. One can show that $Hom_A(-, -)$ is a bifunctor and that the tensor functor is left adjoint to an appropriate Hom-functor. For example, we have natural bijection of sets:
$$\begin{array}{ccc} Hom_{A,C}(M\otimes N, P)&\cong& Hom_{B,C}(N, Hom_A(M, P))\\
&\cong& Hom_{A,B}(M, Hom_C(N, P)).\end{array} $$

\begin{example} The category $A-\mathcal{S}et$ is naturally equivalent to the category of $(A, \Fun)$-bisets. For any $A$-set $M$, there is a covariant functor $$Hom_A(M, -): A-\mathcal{S}et \longrightarrow \Fun-\cM od.$$
\end{example}

\subsection{Normal morphisms and short exact sequences}

Let $A$ be a monoid as before. In this subsection, we check that Quillen's Q-construction can also be applied on $A-\mathcal{S}et$ to defined algebraic K-theory of $A$. Recall that an $A$-set $P$ is projective if the induced map $Hom(P, M) \to Hom(P,N)$ is onto whenever $M\to N$ is onto. One can prove as in the usual case of $A$-acts that $P$ is projective if and only if $A= \bigvee e_iA$ where $e_i^2=e_i$ is an idempotent in $A$. So a projective $A$-set is a retract but in general not a direct summand of a free $A$-set.

Now we define short exact sequences in $A- \cM od$ and make $A- \cM od$ into a {\it quasi-exact category} as introduced in \cite{D06}. Quasi-exact categories are non-additive version of the exact categories of Quillen \cite{Quillen341}. Let \begin{equation}\label{TheSequence} 0\longrightarrow M \stackrel{i}\longrightarrow N \stackrel{j}\longrightarrow K \longrightarrow 0 \end{equation} be a sequence in $A-\mathcal{S}et$.  As discussed in \cite{COR}, to obtain the algebraic K-theory of $A$ using $A$-sets, we can not define (\ref{TheSequence}) to be exact simply by requiring kernel equals image at all terms. We also need to require $i$ and $j$ to be {\it normal}.

\begin{definition} A morphism $f: M\to N$ of $A$-sets is called normal if $f$ is injective when restricted on $M-\ker(f)$. The sequence (\ref{TheSequence}) is called an admissible short exact sequence
if kernel equals image at each term and $i$ and $j$ are normal morphisms. A morphism in $A-\mathcal{S}et$ is called an admissible monomorphism (epimorphism) if it appears as the map $i$ ($j$) in some admissible short exact sequence where all $A$-sets are projective.
\end{definition}

An admissible short exact sequence $$ 0\longrightarrow M \stackrel{i}\longrightarrow N \stackrel{j}\longrightarrow K \longrightarrow 0$$  is called {\it splitting exact} if there exists $s: K \to N$ such that $j\circ s$ is the identity map on $K$. It is easy to check that an admissible short exact sequence is splitting exact if and only it is isomorphic to the standard short exact sequence
$$ 0\longrightarrow M  \longrightarrow M\vee K \longrightarrow K \longrightarrow 0.$$

\begin{proposition}\label{prop: ShortExactSequencesAreSplitting} The admissible short exact sequence $$ 0\longrightarrow M \stackrel{i}\longrightarrow N \stackrel{j}\longrightarrow K \longrightarrow 0$$ is splitting exact if $K$ is projective.
\end{proposition}
\begin{proof} The section $s: K \to N$ exits because $K$ is projective.
\end{proof}

Let $Proj(A)$ (resp. $Vec(A)$) denote the full subcategory of $A-\mathcal{S}et$ which contains only the finitely generated projective (resp. free) $A$-sets. Let $M$ and $N$ be objects in $Vec(A)$. If we choose an ordered set of basis for $M$ and $N$, then homomorphisms from $M$ to $N$ are one to one correspondent to $m\times n$ matrices over $A$ where $m$ and $n$ are the rank of $M$ and $N$. Composition of morphisms coincides with the matrix multiplications defined in Lemma \ref{lemma:MatrixMultiplicationOverMonoids}.

\begin{proposition}\label{prop:PropertiesOfAdmissibleMonosAndEpis}
\begin{enumerate}
      \item[(1)] The collection of admissible monomorphisms (epimorphisms)  is closed under isomorphisms. That is, for any given commutative diagram in $Proj(A)$
       $$\xymatrix{
 M \ar[r]^f \ar[d]_\cong  & N\ar[d]^\cong \\
 M'\ar[r]^{f'}  & N'
          }$$ if the vertical morphisms are isomorphisms, then $f$ is an admissible monomorphism (epimorphism) if and only $f'$ is.
\item[(2)] The collection of admissible monomorphisms (epimorphisms) is closed under composition.
\item[(3)] The collection of admissible monomorphisms (epimorphisms) is stable under base change along admissible epimorphisms (monomorphisms). That is, given a diagram $K\stackrel{f}\longrightarrow M \stackrel{g}\longleftarrow N$ of projective $A$-sets, if $f$ is an admissible monomorphism and $g$ is an admissible epimorphism, then the fiber product $K\times_M N$ exists in $A -\cM od$ and it is projective. The natural map $K\times_M N \to K$ is an admissible epimorphism and the natural map $K\times_M N \to N$ is an admissible monomorphism.
\end{enumerate}
\end{proposition}

\begin{proof} (1) is obvious.

Let $i: M\to N$ and $f: N\to P$ be admissible monomorphisms. By Proposition \ref{prop: ShortExactSequencesAreSplitting}, we can assume that $N=M\vee M'$, $P=N\vee N'$ and $i$ and  $f$ are the inclusions.  So $f\circ i$ is the inclusion $M \to P=M\vee M' \vee N'$.  We see that $M'$ and $N'$ are projective $A$-sets since $P$ is. So $f\circ i$ fits into the following admissible short exact sequence of projective $A$-sets and hence is an admissible monomorphism.
$$0\longrightarrow M \stackrel{f\circ i}\longrightarrow P \longrightarrow M'\vee N' \longrightarrow 0. $$
Similarly, one shows that composition of admissible epimorphisms is again an admissible epimorphism. This proves (2).

As we explained above, the diagram $K\stackrel{f}\longrightarrow M \stackrel{g}\longleftarrow N$ in (3) is isomorphic to
$K\stackrel{i}\longrightarrow K\vee K '\stackrel{j}\longleftarrow K\vee K' \vee M'$ for suitable projective $A$-sets $K'$ and $M'$, where $i$ is the inclusion and $j$ is the projection. One checks easily that the following diagram is a pull back diagram in $A-\mathcal{S}et$, where $j'$ is the natural projection and $i'$ is the inclusion. $$\xymatrix{ K\vee M'\ar[r]^-{i'} \ar[d]_{j'} & K\vee K'\vee M' \ar[d]^j\\
K\ar[r]^-i & K\vee K'
}$$
\end{proof}

The above proposition allows us to apply Quillens' Q-construction to $Proj(A)$. We define a new category, denoted as $QProj(A)$, as follows. Objects in $QProj(A)$ are the same objects in $Proj(A)$. A morphism from $M$ to $N$ in $QProj(A)$ is an isomorphism class of diagrams
$$\xymatrix{M & P\, \ar@{->>}[l]_-j \ar@{>->}[r]^i & N }$$ where $j$ is an admissible epimorphism and $i$ is an admissible monomorphism. Composition of this morphism with the morphism $$\xymatrix{N & Q\, \ar@{->>}[l]_-g \ar@{>->}[r]^f & R }$$ is given by
$$\xymatrix{M & P\times_N Q\, \ar@{->>}[l]_-{\alpha}  \ar@{>->}[r]^-{\beta} & N }$$ where $\alpha$ is the composition of the pull back of $g$ with $j$ and $\beta$ is the composition of the pull back of $f$ with $i$. We can assume that $Proj(A)$ is a small category and hence the geometrical realization of the nerve of $QProj(A)$ exists. We use the standard notation $|\mathcal{C}|$ to denote the geometric realization of the nerve of a small category $\mathcal{C}$.

\begin{definition} The algebraic K-groups of $A$ are defined as $$K_i(A)=\pi_{i+1}(|QProj(A)|) \mbox{  where } i\geq 0.$$
\end{definition}

\begin{proposition}\label{prop:ComputationsOnK0} $K_0(A)\cong \Z$ if $Proj(A)=Vec(A)$.
\end{proposition}
\begin{proof} One checks that $Vec(A)$, together with the admissible short exact sequence, is a quasi-exact category in the sense of \cite{D06}. So the desired isomorphism follows from \cite[Theorem 4]{D06}.
\end{proof}

\section{The comparison theorem}\label{TheComparisonTheorem}
Let $A$ be a monoid such that $Proj(A)=Vec(A)$. Let $\mathcal{S}=Iso(Vec(A))$ be the category of finitely generated free $A$-sets together with isomorphisms. $\mathcal{S}$ is a symmetric monoidal groupoid under direct sums. Recall that the Grothendieck group of $\mathcal{S}$ is defined as the Grothendieck group of the abelian monoid $\pi_0(|\mathcal{S}|)$. We claim that $K_0(\mathcal{S})\cong \mathbb{Z}$. This amounts to saying that any $M\in Vec(A)$ has a well defined rank, which is proved in the following lemma.
\begin{lemma} Let $M$ be a finitely generated free $A$-set. If $$M=\bigvee\limits_{i=1}^n Ax_i=\bigvee\limits_{j=1}^m Ay_j$$ then $m=n$ and for each $1\leq i \leq n$, there exists a unique $j$ and an invertible elements $a\in A$ such that $x_i=a y_j$.
\end{lemma}
\begin{proof} Consider $x_1\in M =\bigvee\limits_{j=1}^m Ay_j.$ Then $x_1\in Ay_j$ for a unique $j$. This also implies that $x_1=a y_j$ for some $a\in A$. Similarly, $y_j=bx_i$ for some unique $i$ and $b\in A$. Now we see that $x_1=a y_j=abx_i$, which implies that $ab=1$ and $i=1$. We also have $y_j=bx_i=bx_1=bay_j$, which implies $ba=1$. So $a$ is invertible. Now an induction on $n$ completes the proof.
\end{proof}

The proof of $``Q=+"$ theorem of the algebraic K-theory of rings can be found in \cite{Grayson, Srinivas}. Since $Vec(A)$ is splitting exact by Proposition \ref{prop: ShortExactSequencesAreSplitting}, the same proof can be applied without any modifications to prove the following theorem. We refer \cite{Srinivas} for details on the localization category $\mathcal{S}^-\mathcal{S}$ associated to any symmetric moniodal category $\mathcal{S}$.

\begin{theorem}{\bf( Q=+ Theorem)} Let $A$ be a monoid such that projective $A$-sets are free.  There is a natural homotopy equivalences $$\mathbb{Z}\times BGL(A)^+ \longrightarrow |\mathcal{S}^-\mathcal{S}|$$ where the second space is weakly homotopy equivalent to $ \Omega |QVec(A)|.$
\end{theorem}

The following fact is well known. We restate it using the language of K-theory of monoids and outline a proof for the reader's convenience.
\begin{corollary} $K_i(A)\cong \pi_i^s(BG_+)$ where $G$ is the group of invertible elements in $A$ and the subscript $+$ means a disjoint base point. In particular, $\pi_2^s(BG_+)$ is the kernel of the universal central extension of $E(A)$.
\end{corollary}
\begin{proof} Using Segal's machinery \cite{Segal}, one proves that $|\mathcal{S}^-\mathcal{S}|$ is weakly homotopy equivalent to $\Omega^\infty\Sigma^\infty BG_+$ because both are the homotopy theoretic group completion of the topological monoid $|\mathcal{S}|$. So we have
$$K_i(A)=\pi_i(\Omega |QVec(A)|)\cong \pi_i(|\mathcal{S}^-\mathcal{S}|)\cong \pi_i \Omega^\infty\Sigma^\infty BG_+ =\pi_i^s (BG_+).$$
\end{proof}

One of the main results of this paper is to show that algebraic K-theory of $\mathbb{F}_1$-schemes are homotopy invariant. We recall from \cite{COR} the definition of algebraic K-theory of $\mathbb{F}_1$-schemes. In particular, we recall that $K_i(Spec(A))$ coincides with $K_i(A)$ as we defined above, where $A$ is an abelian monoid and $Spec(A)$ is the spectrum of $A$ which is an affine $\mathbb{F}_1$-scheme. We also recall from \cite{CC09} that the affine line over $\mathbb{F}_1$, denoted as $\mathbb{A}^1_{\mathbb{F}_1}$, is defined as $Spec(\mathbb{F}_1[x])$ where $\mathbb{F}_1[x]$ is the monoid $\{-\infty, 0, 1, 2, \cdots \}$ under the obvious addition rules. The product scheme $Spec(A)\times \mathbb{A}^1_{\Fun}$ can be shown as $Spec(A[x])$ where $A[x]$, also denoted as $A\underset{\Fun}\otimes \Fun[x]$,  is the following quotient monoid
$$A[x]=A\wedge \mathbb{F}_1[x]=\frac{A \times \mathbb{F}_1[x]}{(A\times \{-\infty\})\cup (\{0_A\} \times \Fun[x])}. $$
\begin{theorem}\label{thm:IsoOnInvertibleElementsInducesIsoOnKGroups} Let $A$ and $B$ be monoids over which projective sets are free. Let $f: A \to B$ be a morphism of monoids which induces isomorphisms on the groups of invertible elements $f: A^\times \to B^\times.$ Then the induced map on K-groups are isomorphisms.
\end{theorem}
\begin{proof} By Proposition \ref{prop:ComputationsOnK0} we have $K_0(A)\cong \mathbb{Z}\cong K_0(B)$ given by the rank of free sets. For $i\geq 1$, the isomorphism $K_i(A)\cong K_i(B)$ follows easily from the +-construction because $GL(A)\cong GL(B)$.
\end{proof}

\begin{corollary} Algebraic K-theory is homotopy invariant on the affine $\Fun$-schemes $\Spec(A)$, where $A$ is a monoid such that projective $A$-sets are free.
\end{corollary}
\begin{proof} This statement follows from Theorem \ref{thm:IsoOnInvertibleElementsInducesIsoOnKGroups}. The natural map $A\to A[x]$ induces isomorphism on the groups of invertible elements and the set of idempotents. In fact, an element $(a, n)\in A[x]$ is invertible if and only if $a$ is invertible in $A$ and $n=0$.  $(a, n)$ is a nonzero idempotent if and only if $a$ is a nonzero idempotent in $A$ and $n=1$.
\end{proof}

\section{Universal central extensions of $E(G_*)$}\label{UniversalCentralExtension}
In this section, we let $G$ denote a finitely generated abelian group. Recall from Example \ref{example:Monoids} that $G_*$ is the monoid $\{*\}\cup A$ where $*$  is the  zero element in $G_*$. It follows from definition that there is a splitting exact sequence of groups
$$\xymatrix{1\ar[r] & \bigoplus\limits_{i=2}^\infty G \ar[r]^u &  E(G_*) \ar[r]^-t & A_\infty \ar@/^.5pc/[l]^-s \ar[r]& 1 }$$ where
$A_\infty=E(\Fun)$ is the union of all finite even permutations, $u((g_i)_{i=2}^\infty)=Diag((g_2g_3\cdots)^{-1}, g_2, g_3, \cdots)$ and $t(\sigma D)=\sigma.$ So $E(G_*)$ is a semidirect product $ \bigoplus\limits_{i=2}^\infty G \rtimes A_\infty$. We denote the associated $A_\infty$-action on $\bigoplus\limits_{i=2}^\infty G$ by $$\lambda: A_\infty\longrightarrow Aut(\bigoplus\limits_{i=2}^\infty G). $$

Since $E(G_*)$ is perfect,  the universal central extension of $E(G_*)$ exists  which we denote by $St(G_*)$. The kernel of $St(G_*)\to E(G_*)$ is the group $K_2(G_*)$ which is also the group $\pi_2^s(BG_+)$. In this section, we want to give an explicit construction of the group $St(G_*)$ where $G$ is cyclic.

Let $St$ denote $St(\Fun)$. The universal central extension of $E(\Fun)=A_\infty$ is well known as
$$ 1 \longrightarrow \Z/2 \longrightarrow St \longrightarrow A_\infty \longrightarrow 1. $$ In particular, $K_2(\Fun)\cong \Z/2$ which is consistent with the fact that $\pi_2^s(S^0)=\Z/2$ where $S^0$ is the point together with a disjoint base point. We also recall from \cite[Section 4.2]{Hatcher} that $\pi_1^s(S^0)=\Z/2$ and $\pi_3^s(S^0)=\Z/24$.

$St$ acts on $\bigoplus\limits_{i=2}^\infty G$ via the action $\lambda \circ (St\to A_{\infty})$, so we have a semidirect product $\bigoplus\limits_{i=2}^\infty G\rtimes St$. One checks that the obvious map
 $$\bigoplus\limits_{i=2}^\infty G\rtimes St \longrightarrow \bigoplus\limits_{i=2}^\infty G\rtimes A_\infty =E(G_*)$$
 is a central extension. So we have a morphism from $St(G_*)$
to $\bigoplus\limits_{i=2}^\infty G\rtimes St$ and hence a morphism
$$\tau: St(G_*)\longrightarrow St.$$ Let $\pi: St(G_*)\to E(G_*)$ be the projection. Since $A_\infty$ is a subgroup of $E(G_*)$, the inverse image $\pi^{-1}(A_\infty)$ is defined. The projection from $\pi^{-1}(A_\infty)$ to $A_\infty$ is a central extension. So we get a morphism from $St$ to $\pi^{-1}(A_\infty)$ and hence a morphism
$$\beta: St\longrightarrow St(G_*).$$

\begin{lemma}\label{lemma:StSplitsSt(G)} The composition $\tau\circ\beta$ is the identity morphism of $St.$
\end{lemma}
\begin{proof} It follows from the definition of $\tau$ and $\beta$ that $\tau\circ\beta$ is compatible with the projection $St\to A_\infty$. So it is the identity map.
\end{proof}

It follows from Lemma \ref{lemma:StSplitsSt(G)} that $St(G)$ is a semidirect product of $St$ with a normal group $N(G)=\ker(\tau)$. Moreover, there is a commutative diagram of splitting extensions
$$\xymatrix{
1 \ar[r] & N(G) \ar[r] \ar[d]^-\pi                           & St(G_*)\ar[r]^-{\tau} \ar[d]^-{\pi} & St \ar[r]\ar[d] \ar@/_-.5pc/[l]^-\beta      & 1  \\
1 \ar[r] & \bigoplus\limits_{i=2}^\infty G \ar[r]   & E(G_*)\ar[r]^t        & A_\infty \ar[r]  \ar@/_-.5pc/[l]^-s     & 1
}$$
We also denote the left vertical map by $\pi$.

To proceed further, we need to recall the Atiyah-Hirzebruch spectral sequence for stable homotopy groups of a CW-complex $X$
$$E^2_{p,q}=H_p(X, \pi_q^s(S^0))\Longrightarrow \pi_{p+q}^s(X_+).$$
The differentials arriving at the column $p=0$ are all trivial because the natural maps $ S^0\to X_+ \to S^0$ imply that $\pi^S_*(S^0)$ is a
 direct summand of $\pi^S_*(X_+).$ Let $\pi^S_*(X)$ denote the natural quotient $\pi^S_*(X_+)/\pi^S_*(S^0)$ and  let $d$ denote the differential $$d: \, E^2_{3,0}=H_3(G, \Z)\longrightarrow E^2_{1,1}=H_1(G, \Z/2).$$
The following observation is immediate from the spectral sequence and the above discussion.

\begin{lemma}\label{lemma:FromTheSpectralSequence} $\pi^S_*(X_+)=\pi^S_*(S^0)\oplus \pi^S_*(X)$ and there is a natural exact sequence
$$0\longrightarrow Coker(d)\longrightarrow \pi^S_2(X)\longrightarrow H_2(X, \mathbb{Z})\longrightarrow 0.$$
In particular, if  $X=BG$ where $G$ is cyclic, then $K_2(G_*)=\pi_2^S(BG_+)=\Z/2\oplus Coker(d)$ and $Coker(d)$ is a quotient of  $\Z/2$. \end{lemma}

\begin{lemma}\label{lemma:KernelAndCokernelOfPi} Let $G$ be a cyclic group. The morphism $$\pi: N(G)\longrightarrow \bigoplus\limits_{i=2}^\infty G$$ is onto and its  kernel equals $Coker(d)$.
\end{lemma}
\begin{proof} The fact that $\pi: N(G)\to \bigoplus\limits_{i=2}^\infty G$ is onto can be checked by diagram chasing on the diagram after Lemma \ref{lemma:StSplitsSt(G)}, using the fact that $\beta(St)\subset \pi^{-1}(A_\infty)$. The second statement follows from Lemma \ref{lemma:FromTheSpectralSequence}.
\end{proof}

\begin{proposition}\label{prop:UCEforZmodOddNumber} Let $G=\Z/d$ where $d$ is an odd integer, then $K_2(G_*)$ is naturally isomorphic to $K_2(\Fun)=\Z/2$ and $St(G_*)\cong \bigoplus\limits_{i=2}^\infty G \rtimes St$. The following group extension is the universal central extension of $E(G_*)$.
$$\xymatrix{
1 \ar[r] & \Z/2  \ar[r]   & \bigoplus\limits_{i=2}^\infty G \rtimes St \ar[r]  & E(G_*) \ar[r] & 1
}$$
\end{proposition}
\begin{proof} If $d$ is odd, $H_1(G, \Z/2)=0$ and hence $Coker(d)=0$. This proves that $K_2(G_*)=\Z/2$. Together with Lemma \ref{lemma:KernelAndCokernelOfPi}, this also proves  that $N(G)\cong \bigoplus\limits_{i=2}^\infty G$. So $St(G_*)\cong \big(\bigoplus\limits_{i=2}^\infty G\big) \rtimes St$ by Lemma \ref{lemma:StSplitsSt(G)}. In particular, the universal central extension of
$E(G_*)$ is given by the desired sequence. \end{proof}

\begin{remark}
From the commutative diagram after Lemma \ref{lemma:StSplitsSt(G)}, it is easy to check that the action of $St$ on $\bigoplus\limits_{i=2}^\infty G $ is the one we defined before:
$$\lambda \circ (St\to A_\infty): \,\, St \longrightarrow Aut(\bigoplus\limits_{i=2}^\infty G ).$$
\end{remark}

Now we introduce a group by generators and relations. For any integer $d$, let $M(\Z/d)$ be the group with the following presentations.
\begin{center}
\begin{tabular}{rl}
Generators:& $\alpha;\, X_2, X_3, X_4\cdots$ \\
Relations:& $[\alpha, X_i]=1$ where $2\leq i$ \\
                 & $[X_i, X_j]=\alpha$ where $2\leq i \neq j$\\
                 & $X_i^d=1$ where $2\leq i$.
\end{tabular}
\end{center}

\begin{remark} (1) If $d=0$, then $\Z/d=\Z$ and the last relation $X_i^0=1$ is by convention automatically true.

(2) It follows from the second relation that $\alpha^2=1$. If $d$ is odd, one can show  that $\alpha=1.$ So $M(\Z/d)=N(\Z/d)=\bigoplus\limits_{i=2}^\infty \Z/d$ if $d$ is odd.

(3) Every element in $M(\Z/d)$ can be written as $\alpha^rX_{i_1}^{e_{i_1}}X_{i_2}^{e_{i_2}}\cdots X_{i_n}^{e_{i_n}}$ where $n\geq 0$, which we call the standard form. The exponents $e_1, e_2,\cdots e_n$ are unique. It follows that $\bigoplus\limits_{i=2}^\infty \Z/d$ is a quotient of $M(\Z/d)$ with kernel the subgroup generated by $\alpha$. We can not discuss the uniqueness of $r$ for the moment because we do not know the order of $\alpha$ yet.
\end{remark}

\begin{lemma}\label{lemma:N(G)MapsOntoN(G)}There is an onto morphism $f: N(\Z/d)\to M(\Z/d)$.
\end{lemma}
\begin{proof}
We first define a semidirect product $M(\Z/d)\rtimes St$.  Let $A_\infty$ act on positive integers naturally. For any $\sigma\in A_\infty$, one checks that $\sigma(\alpha)=\alpha$ and
$$\sigma(X_i)=\begin{cases}
X_{\sigma(i)} & \mbox{ if } \sigma(1)=1\\
X_{\sigma(1)}^{-1} & \mbox{ if } \sigma(i)=1\\
X_{\sigma(1)}^{-1}X_{\sigma(i)} & \mbox{ otherwise}
\end{cases}$$
defines an automorphism of $ M(\Z/d)$. Moreover, it defines an $A_\infty$-action on $M(\Z/d)$. Let $St$ act on $M(\Z/d)$ through the natural map $St\to A_\infty$, then we have a semidirect product  $M(\Z/d)\rtimes St$. Since the $St$-action on $M(\Z/d)$ is compatible with the $A_\infty$-action on $\bigoplus\limits_{i=2}^\infty \Z/d$, the natural maps $$M(\Z/d)\stackrel{\pi'}\longrightarrow \bigoplus\limits_{i=2}^\infty \Z/d \mbox{ and }
St \stackrel{\pi}\longrightarrow A_\infty$$ define a morphism $\pi': M(\Z/d)\rtimes St\to E(\Z/d_*)$ which is a central extension. So there is an  induced map $$f:St(\Z/d_*)=N(\Z/d)\rtimes St \longrightarrow M(\Z/d)\rtimes St.$$
We now have the following commutative cylinder, without the two vertical dotted maps.
 $$\xymatrix{
 N(\Z/d)\ar[rr]\ar[dr]_-\pi \ar@{.>}[dd]^f & & St(\Z/d_*) \ar[rr]^\tau \ar[dr]_-\pi \ar'[d][dd]^f&      & St \ar[dr]_\pi \ar@{.>}'[d][dd]^g \ar@/^1pc/[ll]^-\beta\\
               &\bigoplus\limits_{i=2}^\infty \Z/d \ar[rr] &                 & E(\Z/d)\ar[rr] && A_\infty \\
M(\Z/d)\ar[ur]^{\pi'}\ar[rr] & & M(\Z/d) \rtimes St \ar[ur]^{\pi'}\ar[rr]^-{\tau'} &&St \ar[ur]^\pi
 }$$

The right vertical map $g$ is defined as the composition map
 $$St\stackrel{\beta}{\longrightarrow}St(\Z/d_*)\stackrel{f}\longrightarrow M(\Z/d)\rtimes St\stackrel{\tau'}\longrightarrow St.$$
One checks that the right triangle commutes, i.e., $\pi\circ g=\pi$. So $g$ is the identity map.

Next we show that the right vertical square commutes, i.e., $\tau=g\circ \tau =\tau'\circ f.$ We recall the semidirect product $\bigoplus\limits_{i=2}^\infty \Z/d \rtimes St$ and the projection $\bigoplus\limits_{i=2}^\infty \Z/d \rtimes St \to E(\Z/d_*)$ which we now denote by $\pi''$. The various maps defined before fit into the following commutative diagram, where the dotted arrow $f$ is not included.
$$\xymatrix{ St(\Z/d_*)\ar[d]^\tau \ar[dr]_-{f''} \ar[drr]^\pi \ar@{.>}@/_2pc/[dd]_f & & \\
 St & \bigoplus\limits_{i=2}^\infty \Z/d \rtimes St \ar[l]_{\tau''} \ar[r]^{\pi''} & E(\Z/d_*)\\
 M(\Z/d)\rtimes St \ar[u]_{\tau'} \ar[ur]^-{f'} \ar[urr]_{\pi'} & &
}$$
In the diagram, $f'=\pi'\rtimes Id$ is well defined. $f''$ is the unique map induced by $\pi''$. Since $\pi''\circ f'\circ f= \pi'\circ f=\pi$, we see that $f'\circ f= f''$ by the universal property of universal central extension. This implies $\tau'\circ f= \tau''\circ f'\circ f=\tau'' \circ f''=\tau$.

Since $\tau'\circ f= g\circ \tau$, the commutative cylinder implies that $f$ restricts to a map $N(\Z/d)\to M(\Z/d)$, for which we denoted also by $f$. Obviously, the left triangle in the cylinder commutes because $\pi$ factors through $f$:
$$\xymatrix{N(\Z/d) \ar[r]_f \ar@/^1pc/[rr]^\pi &M(\Z/d) \ar[r]& \bigoplus\limits_{i=2}^\infty \Z/d.  }$$

Since $\pi$ is onto, for any $i\geq 2$ we have $X_i$ or $\alpha X_i$ belongs to $Img(f)$. Since $[X_i\alpha^s, X_j\alpha^t]=\alpha$ for any integers $s$ and $t$, we see $\alpha\in Img(f)$. In turn, this implies both $X_i\alpha$ and $X_i$ are in $Img(f)$. So $f$ is onto.
\end{proof}

\begin{lemma} \label{lemma:M(G)equalsN(G)ForGBeingZ} $f: N(\Z)\to M(\Z)$ is an isomorphism.
\end{lemma}
\begin{proof} One can compute $\pi^S_2(B\Z_+)=\pi^S_2(S^1_+)\cong \Z/2\oplus \Z/2$ easily from the Atiyah-Hirzebruch spectral sequence. So $Coker(d)=\langle a \rangle\cong \Z/2$ where $a$ denotes the unique generator. There is a commutative diagram
$$\xymatrix{1\ar[r] & \langle a \rangle \ar[r]\ar[d]^f & N(\Z)  \ar[r]^-\pi \ar[d]^f& \bigoplus\limits_{i=2}^\infty \Z \ar[r] \ar@{=}[d] &1\\
            1\ar[r] & \langle\alpha\rangle\ar[r] & M(\Z)  \ar[r]& \bigoplus\limits_{i=2}^\infty \Z \ar[r] &1  }$$
where the top row is given by Lemma \ref{lemma:KernelAndCokernelOfPi} and $f$ is the onto map defined in Lemma \ref{lemma:N(G)MapsOntoN(G)}. It is easy to see that $f(a)=\alpha$ because $f$ is onto and $\alpha$ has order at most 2.  To show that $f$ is an isomorphism, we only need to show that $f(a)=\alpha \neq 1$ in $M(\Z)$.

Let $y_2, y_3\cdots $ be the standard basis of $\bigoplus\limits_{i=2}^\infty \Z$ and let $Y_i$ be a fixed preimage in $\pi^{-1}(y_i)$. For any integers $i,j,s,t\geq 2$ such that $i\neq j$ and $s\neq t$, it is routine to show that $[Y_i, Y_j]=[Y_s, Y_t]$. That is, chose some $\sigma\in A_\infty$ such that $\sigma(1)=1, \sigma(i)=s$ and $\sigma(j)=t$. Let $\Sigma \in St$ be a preimage of $\sigma$, then one checks that $\Sigma Y_i\Sigma^{-1}$ and $\Sigma Y_j\Sigma^{-1}$ differs from $Y_s$ and $Y_t$ by central elements. So $[Y_i, Y_j] =[Y_s, Y_t]$. Let $b$ denote this common value. Obviously, $b\in \ker(\pi)$ is a central element. So  every element in the semidirect product $St(\Z_*)$ can be written as $(b^rY_1^{e_1}Y_1^{e_1}Y_2^{e_2}\cdots, u)$ where $u\in St$ and the powers $e_1, e_2,\cdots$ are unique. So $b$ generates the kernel of the map $N(\Z) \to \bigoplus\limits_{i=2}^\infty \Z $, which implies that $b=a$. It follows that there is a well defined group homomorphism
$$g: M(\Z) \longrightarrow N(\Z)$$ which sends $\alpha $ to $a$ and $X_i$ to $Y_i$. Since $a\neq 1$, so $\alpha \neq 1$ and hence completes the proof. \end{proof}

\begin{remark} The above argument fails if we replace the group $\Z$ by $\Z/d$ where $d$ is a nonzero even number. This is because, in this case, the element $(b^rY_1^{e_1}Y_1^{e_1}Y_2^{e_2}\cdots, 1)$ is in $\ker(\pi)$ if and only if each power $e_i$ is a multiple of $d$. It could happen that $Y_i^d \neq 1$ while $b=1$ which implies we can not define the map $g$ as in the case $d=0$.  However, we can see later from Theorem \ref{thm:UniversalCentralExtensionOfE(G)} that $Y_i^d=1$ and $b\neq 1$.
\end{remark}

\begin{lemma}\label{lemma:AlphaIsNotZero} If $d$ is even, then $\alpha \neq 1$ in $M(\Z/d)$.
\end{lemma}
\begin{proof} By Lemma \ref{lemma:M(G)equalsN(G)ForGBeingZ}, we see that $\alpha \neq 1$ in $M(\Z)$. If $d$ is  a nonzero even number, we consider the map $\rho: M(\Z)\to M(\Z/d)$ which sends $\alpha $ to $\alpha$ and $X_i$ to $X_i$. The kernel of $\rho$ is the normal subgroup generated by $\{X_i^d|2\leq i\}$. So an element $x$ is in the kernel of $\rho$ if and only if $x$ can be written as
$$x= \prod\limits_{l=1}^Na_l X_{i_l}^{^{_+}_{^-}d} a_l^{-1}$$
where $a_l\in M(\Z)$. Rewrite $x$ in the standard form $x=\alpha^rX_{i_1}^{e_{i_1}}X_{i_2}^{e_{i_2}}\cdots X_{i_n}^{e_{i_n}}$, one sees that $r$ is a multiple of $d$. This shows that $r$ is always even. In particular, $\alpha$ is not in the kernel of $\rho$. So $\alpha$ is not the identity in $M(\Z/d)$ if $d$ is even.
\end{proof}

\begin{theorem}\label{thm:UniversalCentralExtensionOfE(G)} The universal central extension of $E(\Z/d_*)$, where $d$ is even, is given by
$$1\longrightarrow \Z/2\oplus \Z/2 \longrightarrow M(\Z/d)\rtimes St {\longrightarrow} E(\Z/d_*) \longrightarrow 1.$$ In particular, $K_2(\Z/d_*)=\pi^S_2(B\Z/d_+)\cong \Z/2\oplus \Z/2$.
\end{theorem}
\begin{proof} It is enough to show that $f: N(\Z/d) \to M(\Z/d)$ is an isomorphism. The case $d=0$ is considered in Lemma \ref{lemma:M(G)equalsN(G)ForGBeingZ}. In general, we consider the following commutative diagram
$$\xymatrix{1\ar[r] & Coker(d) \ar[r]\ar[d]^f & N(\Z)  \ar[r]\ar[d]^f& \bigoplus\limits_{i=2}^\infty \Z \ar[r] \ar@{=}[d] &1\\
            1\ar[r] & \langle\alpha\rangle\ar[r] & M(\Z)  \ar[r]& \bigoplus\limits_{i=2}^\infty \Z \ar[r] &1.  }$$
Since $\alpha \neq 1$ in $M(\Z/d)$ by Lemma \ref{lemma:AlphaIsNotZero}, we have $\langle\alpha\rangle\cong\Z/2$. Since $Coker(d)$ is a quotient of $\Z/2$, the onto morphism $Coker(d) \to  \langle\alpha\rangle$ has to be an isomorphism. This completes the proof.
\end{proof}

The following result is not surprising to experts. For lacking of a direct reference, we give a proof.
\begin{corollary}\label{corr:K2OfFGAbelainGroups} Let $G$ be a finitely generated abelian group.  In the Atiyah-Hirzebruck spectral sequence of the stable homotopy groups for $BG_*$, the differential map
$$d: E^2_{3,0}=H_3(G, \Z)\longrightarrow E^2_{1,1}=H_1(G, \Z/2)$$  is the zero map.  Consequently, $K_2(G_*)\cong \pi_2^s(BG_+)\cong \Z/2 \oplus G/2 \oplus H_2(G, \Z).$
\end{corollary}
\begin{proof} If $G$ is cyclic, $d=0$ by Proposition \ref{prop:UCEforZmodOddNumber} and Theorem \ref{thm:UniversalCentralExtensionOfE(G)}.  In general, let $G=G_1\oplus G_2 \oplus \cdots \oplus G_n$ be the decomposition of $G$ into cyclic groups. Since the Atiyah-Hirzebruch spectral sequence is natural with respect to continuous maps and there is a splitting morphism of groups $$G_i\stackrel{\tau_i}\longrightarrow G \stackrel{p_i}\longrightarrow G_i,$$ we have the following commutative diagram of short exact sequences from Lemma \ref{lemma:FromTheSpectralSequence}
$$\xymatrix{
0 \ar[r] & Coker(d)\ar[r]\ar@<.5ex>[d]^p                                     & \pi^S_2(BG)\ar[r]\ar@<.5ex>[d]^{p}  & H_2(G, \Z) \ar[r]& 0\\
0 \ar[r] & \bigoplus\limits_{i=1}^n G_i/2 \ar[r]^-\cong \ar@<.5ex>[u]^{\tau} &  \bigoplus\limits_{i=1}^n \pi^S_2(BG_i)\ar[r]\ar@<.5ex>[u]^{\tau}  &     0      \ar[r]& 0
}$$ where $\tau=\coprod{\tau_i}_* $ and $p= \prod {p_i}_*$. One checks that the composition  $p\circ \tau$ is the identity map. So $\tau$ is an injective map between finite sets. Since $\bigoplus\limits_{i=1}^n G_i/2 \cong G/2$ and $Coker(d)$ is a quotient of $G/2$, $\tau$ is an isomorphism. So $Coker(d)\cong G/2$ and the top exact sequence in the diagram splits.
\end{proof}

We end this section by computing $\pi^S_2(BG+)$ for some nonabelian groups $G$. The proof is the same as the proof of Corollary \ref{corr:K2OfFGAbelainGroups} and hence skipped.
\begin{corollary} Let $n\geq 2$. Let $\Sigma_n$ be the symmetric group on $n$ objects, $A_n$ the group of even permutations in $\Sigma_n$ and $\mathbb{F}_q$ the field of $q$ elements.
\begin{enumerate}\item[1.] $\pi_2^S(B{A_n}_+)=\Z/2\oplus H_2(A_n, \Z)$.
\item[2.] $\pi_2^S(B{\Sigma_n}_+)=\Z/2\oplus\Z/2\oplus H_2(\Sigma_n, \Z)$.
\item[3.] $\pi_2^S(BE_n(\mathbb{F}_q)_+)=\Z/2\oplus H_2(E_n(\mathbb{F}_q), \Z)$.
\item[4.] $\pi_2^S(BGL_n(\mathbb{F}_q)_+)=\Z/2\oplus\Z/2\oplus H_2(GL_n(\mathbb{F}_q), \Z)$ if $q$ is odd or $\Z/2\oplus H_2(GL_n(\mathbb{F}_q), \Z)$ if $q$ is even.
\end{enumerate}
\end{corollary}

\end{document}